\documentclass[11pt]{article}
\usepackage[colorlinks=true, urlcolor=blue, linkcolor=red]{hyperref}
\usepackage{epsf}
\usepackage{amsbsy,amsmath}
\usepackage{mathtools}
\usepackage{mathrsfs}
\usepackage{amsfonts}
\usepackage{amssymb}
\usepackage{enumitem}
\usepackage{eucal}
\usepackage{enumitem}
\usepackage{graphics,mathrsfs}
\usepackage{amsthm}
\usepackage{secdot}
\usepackage{xcolor}
\usepackage{color,soul}
\newtheorem{theorem}{Theorem}[section]

\theoremstyle{definition}
\newtheorem{definition}[theorem]{Definition}

\theoremstyle{remark}
\theoremstyle{claim}
\newtheorem{remark}[theorem]{Remark}
\numberwithin{equation}{section}

\numberwithin{equation}{section}
\newsavebox{\savepar}

\begin{document}
	
\title{ \sc  A multiphase eigenvalue problem\\ 
             on a stratified Lie group}		
	\author{\sc Debajyoti Choudhuri$^{a,}$, 
	Leandro S. Tavares$^b$, 
	Du\v{s}an D. Repov\v{s}$^{c,}$\footnote{Corresponding author: dusan.repovs@guest.arnes.si}\\[0.5cm]
	\small{$^{a}$School of Basic Sciences, Indian Institute of Technology Bhubaneswar, Khordha, 752050, Odisha, India.}\\
	\small{{\it Email: dchoudhuri@iitbbs.ac.in}}\\
	\small{$^b$Center of Sciences and Technology, Federal University of Cariri,   Juazeiro do Norte, CE, 63048-080, Brazil.}\\
	\small{{\it Email: leandro.tavares@ufca.edu.br}}\\
\small{$^c$Faculty of Education and Faculty of Mathematics and Physics, University of Ljubljana,}\\
\small{\& Institute of Mathematics, Physics and Mechanics, Ljubljana, 1000, Slovenia.}\\
\small{{\it Email: dusan.repovs@guest.arnes.si}}}

\date{}

\maketitle	
 
\begin{abstract}
We consider a multiphase spectral problem on a stratified Lie group. We prove the existence of an eigenfunction of $(2,q)$-eigenvalue problem on a bounded domain. Furthermore, we also establish a Pohozaev-like identity corresponding to the problem on the Heisenberg group.
\begin{flushleft}
\end{flushleft}
\begin{flushleft}
{\it Math. Subj. Classif. (2020)}:~35J35, 35J60.
\end{flushleft}
\end{abstract}
\section{Introduction}
In this paper, we shall study the following problem
\begin{equation}\label{main}
	\left\{ \begin{array}{ll}-\mathcal{L} u+(-\Delta)^su=\lambda\|u\|_q^{2-q}|u|^{q-2}u~~~\text{on}~\Omega,\\
u=0~~~\text{on}~\mathbb{G}\setminus\Omega,
\end{array} 
	\right.
	\end{equation}
where $\Omega\subset\mathbb{G}$ is a bounded subdomain of a stratified Lie group $\mathbb{G}$. We shall  further assume the following condition on the exponent $q$
$$\text{(C)}:~~~2<Q,~1<q<2^*,$$
where $2^*:=\frac{2Q}{Q-2}$
and
$Q$ is the homogeneous dimension of the group $\mathbb{G}$. This is a new direction of studying the multiphase eigenvalue problem because the problem is
considered on a stratified Lie group. 

The interest towards the eigenvalue problems like the one in problem \eqref{main} is not only restricted to within the mathematical community but is also of interest to  physicists since it has a relation with the spectral optimization theory, bifurcation theory, fluid and quantum mechanics (see L\^{e}  \cite{Le}, Lindqvist  \cite{peter1}). Interested readers  may note that the study of elliptic PDEs involving the $p$-Laplacian operator is of interest in the theory of non-Newtonian fluids,
 both for the case $p \geq 2$ (dilatant fluids) and the case $1<p<2$ (pseudo-plastic fluids), see Astarita-Marrucci  \cite{4}. It is also of geometrical interest for $p \geq 2$, some of which is discussed in Uhlenbeck  \cite{32}. 

As far as the nonlocal elliptic problems are concerned, we refer to Zhao et al. 
\cite{zhao2} 
who considered a nonlocal elliptic problem driven by a nonlinearity, obeying certain conditions. Saoudi et al.
\cite{saoudi1}
proved the existence, multiplicity and regularity of solutions of a nonlocal elliptic PDE, driven by a singular and a power nonlinearity. 
The reader may also check 
Bouabdallah et al. \cite{bouabdallah}
and 
Zhao et al. \cite{zhao2},
to understand the trends in research on nonlocal elliptic PDEs, driven by nonlinearities of various types, which were naturally motivated by the literature in the local case.  For the latter,  we point out for example, Zeddini \cite{Z}.
However, since our paper is purely of mathematical interest, we further refer the reader to Section \ref{s2} for the work due to Dipierro-Valdinoci \cite{SD-EV-0}, to learn more about the theory of phase transitions and the associated mathematics involved. 

 Interesting contributions related to \eqref{main} can be found in the literature, see e.g.,  Alves-Covei  \cite{AC}, Corr\^{e}a et al.  \cite{CFL}.  For example, in  Alves and Covei  \cite{AC}, the sub-supersolution method  was  applied to establish
  the existence of solutions for 
\begin{equation}\label{alves-covei-eq}
	\left\{ 
	\begin{array}{rcl}
		-a\left( \int_{\Omega }u\right) \Delta u & = & h_{1}(x,u)f\left(
		\int_{\Omega }|u|^{p}\right) +h_{2}(x,u)g\left( \int_{\Omega }|u|^{r}\right)
		\;\;\mbox{on}\;\;\Lambda , \\ 
		u & = & 0\;\;\mbox{on}\;\;\partial \Lambda ,%
	\end{array}%
	\right. 
\end{equation}
where $\Lambda \subset \mathbb{R}^n$ is a bounded smooth domain, and $a,f,g,h_{i}$ ($i=1,2$) are given functions with sufficient regularity.  In the case of the $p-$Laplacian, Corr\^{e}a et al.  \cite{CFL}  combined the sub-supersolutions method with a classical theorem due to  Rabinowitz  \cite{rabinowitz}. As a result, they were able to prove  the existence of solutions for the  quasilinear problem 
$$
\left\{ 
\begin{array}{rcl}
	-\Delta _{p}u & = & \|u\|_{L^{q}}^{\alpha (x)}\;\;\mbox{on}\;\;\Omega , \\ 
	u & = & 0\;\;\mbox{on}\;\;\partial \Omega ,%
\end{array}%
\right. \eqno{(P)}
$$%
where $\alpha $ is a nonnegative function defined on $\bar{\Omega }.$ The works of Arora-R\u{a}dulescu  \cite{radu1}, Bahrouni-R\u{a}dulescu  \cite{radu2}, Garain et al.  \cite{PG2}, Gou-R\u{a}dulescu  \cite{radu3}, Razani-Behboudi \cite{razani} and Zhang-R\u{a}dulescu  \cite{radu4}, where the multiphase problems were studied, are also of interest here. 

We now state the first main results of the paper:

\begin{theorem}\label{main_result_3}
Let $0<s<1$ and $1<q<2^*$. Assume that $\lambda>0$ and 
 $u\in X\setminus\{0\}$ is an eigenvalue and the corresponding eigenfunction. Then $u$ is bounded. Furthermore, if $u$ is nonnegative on $\Omega$, then $u>0$ on $\Omega$. Moreover, for every relatively compact subset  $\omega \subset \Omega,$ there exists a positive constant $c(\omega)$ such that $u\geq c>0$ on $\omega$.
\end{theorem}
 Our second main result is a Pohozaev-like identity for the Brezis-Nirenberg problem, albeit on a Heisenberg group $\mathbb{H}^n,$  which is a particular type of stratified Lie group.

\begin{theorem}\label{th2}
The following Pohozaev-like identity corresponding to \eqref{main} holds
\begin{align}\label{pohozaev_identity_2'}
\begin{split}
\frac{Q}{2}\int_{\Omega}G(u)dx-\left(\frac{Q-2}{2}\right)\int_{\Omega}|\nabla_{\mathbb{H}^n}u|^2dx&-\left(\frac{Q-2s}{2}\right)\int_{\Omega}|(-\Delta)^su|^2dx\\
=&\frac{1}{2}\int_{\partial\Omega}|\nabla_{\mathbb{H}^n}u|^2\langle Z,\hat{n}\rangle dS.
\end{split}
\end{align} 
\end{theorem}

We complete the introduction by describing the structure of the paper. 
In Section \ref{s2}, we review the fundamentals of stratified Lie groups. We also prove an embedding result in the Lie group setup.
In Section \ref{s22}, we introduce two key operators $A$ and $B,$ which are necessary for the proof of the first main result.
In Section \ref{s3}, we prove the first main result (Theorem \ref{main_result_3}).
In Section \ref{s4}, we establish the second main result (Theorem \ref{th2}). For all fundamental material used in this paper we refer the reader to the comprehensive monograph by Papageorgiou et al.  \cite{PRR}.

\section{Stratified Lie groups}\label{s2}

A quick sneak into the basics of stratified Lie groups may be useful to the reader  (see Choudhuri-Repov\v{s}  \cite{CR1}, Choudhuri et al. \cite{chou1}, Folland et al. \cite{GBF1}, Ghosh et al.  \cite{GVM1}, Montgomery \cite{montgomery1}, and the references therein). This is why we now recall some fundamental definitions from this topics.
\begin{definition}\label{HLG}
	A Lie group $\mathbb{G}$ on $\mathbb{R}^n$, is said to be {\it homogeneous}, if for every $\delta>0$ there exists an automorphism $T_{\delta}:\mathbb{G}\to\mathbb{G}$ defined by 
	$$T_{\delta}(x)=(\delta^{r_1}x_1,\delta^{r_2}x_2,\cdots,\delta^{r_n}x_n),
	\ 
	\hbox{ for every}
	 \ 
	r_i>0,
	 \
	 i=1,2,\cdots,n.
	$$
	 The map $T_{\delta}$ is called a {\it dilation} on $\mathbb{G}$. Here, $x=(x_1,x_2,\cdots,x_n)$.
\end{definition}

The number $n$ represents the {\it topological dimension} of $\mathbb{G}$, whereas the number $M=r_1+r_2+\cdots+r_n$ represents the {\it homogeneous dimension} of the homogeneous Lie group $\mathbb{G}$. We shall denote Haar measure by the symbol $dx$, which is the standard Lebesgue measure on $\mathbb{R}^n$. The following is the definition of a stratified Lie group.
\begin{definition}\label{SLG_defn}
	A homogeneous Lie group $\mathbb{G}=(\mathbb{R}^n,\cdot)$ is called a {\it stratified Lie group} (or a {\it homogeneous Carnot group}) if the following two conditions hold:
	\begin{enumerate}[label=(\roman*)]
		\item The decomposition $\mathbb{R}^n=\mathbb{R}^{n_1}\times\mathbb{R}^{n_2}\times\cdots\times\mathbb{R}^{n_k}$ holds for some natural numbers $n_1,n_2,\cdots,n_k$ such that $n_1+n_2+\cdots+n_k=n$. Furthermore, for every $\delta>0$ there exists a dilation of the form $T_{\delta}(x)=(\delta^1 x^{(1)},\delta^2 x^{(2)},\cdots,\delta^k x^{(k)})$ which is an automorphism of the group $\mathbb{G}$. Here, $x^{(i)}\in\mathbb{R}^{n_i}$ for every $i=1,2,\cdots,k$.
		\item Let $n_1$ be the same as in the above decomposition of $\mathbb{R}^n$, and let $\mathcal{L}_1,\mathcal{L}_2,\cdots,\mathcal{L}_{n_1}$ be the left invariant vector fields on $\mathbb{G}$ such that $\mathcal{L}_i(0)=\frac{\partial}{\partial x_i}|_{0}$ for $i=1,2,\cdots,n_1$. Then the H\"{o}rmander rank condition holds for every $x\in\mathbb{R}^n$, i.e., rank(Lie$\{\mathcal{L}_1,\mathcal{L}_2,\cdots,\mathcal{L}_{n_1}\})=n$. Roughly speaking, the Lie algebra corresponding to the Lie group $\mathbb{G}$ is spanned by the iterated commutators of $\mathcal{L}_1,\mathcal{L}_2,\cdots,\mathcal{L}_{n_1}$.
	\end{enumerate}
\end{definition}

Together, the structure $(\mathbb{R}^n,\cdot,T_{\delta})$ is called a {\it stratified Lie group}. The number $k$ is called the {\it step of the homogeneous Carnot group}. In the case of a stratified Lie group, the homogeneous dimension becomes $$Q=\sum_{i=1}^{k}in_i$$ Let $\mathfrak{G}$ be a Lie algebra associated to a Lie group $\mathbb{G}$. A {\it stratification} of $\mathfrak{G}$ with step $k$ is a direct sum decomposition
$$\mathfrak{G}=V_1\oplus V_2\oplus\cdots\oplus V_k$$
of $\mathfrak{G}$ with the property that $V_k\neq\{0\}$ and $[V_2,V_j]=V_{j+1}$, for every $j=1,2,\cdots,k$, where we set $V_{k+1}=\{0\}$. Here $[V,W]:=\text{span}\{[X,Y]:X\in V, Y\in W\}$.

Throughout the paper, we set $n=n_1$ in Definition \ref{SLG_defn}. The operators $\mathcal{L}$
and  $(-\Delta)^s$ are the Laplacian and the fractional Laplacian, respectively that are defined as follows:
\begin{align}\label{operators}
\begin{split}
\mathcal{L}v&:=\sum_{i=1}^{n_1}X_i^2v,\\
(-\Delta)^sv&:=P.V.\int_{\mathbb{G}}\frac{v(x)-v(y)}{|y^{-1}\cdot x|^{Q+2s}}dy,
\end{split}
\end{align}
where for every $i$ the vector field $X_i$ are left invariant.  
The subgradient is the $n_1$-dimensional vector given by 
$$\nabla_{\mathbb{G}} v(z)=(X_1v,X_2v,\cdots,X_{n_1}v).$$
The operators $\mathcal{L}$ and  $\nabla_{\mathbb{G}}$ are left invariant differential operators.
For each real $\delta$, the naturally associated dilation with a stratified Lie group is given by 
$$T_{\delta}(z)=(\delta^1 z^{(1)},\delta^2 z^{(2)},\cdots,\delta^kz^{(k)}),
\
\hbox{for every}
\
z\in\mathbb{R}^{n_1}\times\mathbb{R}^{n_2}\times\cdots\times\mathbb{R}^{n_k}.
$$
The fractional Sobolev space $W^{s,2}(\Omega)$, $0<s<1$, is defined by 
$$W^{s,2}(\Omega):=\left\{u\in L^2(\Omega):\frac{|v(x)-v(y)|}{|y^{-1}\cdot  x|^{\frac{Q}{2}+s}}\in L^2(\Omega\times\Omega)\right\},$$
and is equipped with the norm 
$$\|v\|_{W^{s,2}}(\Omega):=\left(\int_{\Omega}|u|^2dx+\iint_{\Omega\times\Omega}\frac{|u(x)-u(y)|^2}{|y^{-1}\cdot  x|^{Q+2s}}dxdy\right)^{1/2}.$$
However, 
in order
to study the mixed problem we shall consider the space
$$X:=\{u\in W^{s,2}(\mathbb{G}):u|_{\Omega}\in W_0^{1,2}(\Omega),~u=0~\text{a.e. on}~\mathbb{G}\setminus\Omega\},$$
endowed with the norm
$$\|u\|:=\|\nabla_{\mathbb{G}} v\|_{L^2(\Omega)}+\left\|\frac{v(x)-v(y)}{|y^{-1}\cdot  x|^{\frac{Q+2s}{2}}}\right\|_{L^{2}(\Omega\times\Omega)},$$
which is more appropriate for study.

The continuous and compact embedding theorems also hold (see   Has\l asz-Koskela  \cite[Theorem 8.1]{HK}). We can now state the result collectively as follows.

\begin{theorem}\label{sobolev_embedding}
	Let $\Omega\subset\mathbb{G}$ be a bounded domain and $1\leq p <Q$. Then $W_0^{1,p}(\Omega)$ is continuously embedded
	in $L^q(\Omega)$ for every $1 \leq q \leq p^*:=\frac{Qp}{Q-p}$. Moreover, the embedding is compact for every $1 \leq q < p^*$.
\end{theorem}
We now prove the embedding result along the lines of Buccheri et al. \cite{BSM1}.
\begin{theorem}\label{lemma_on_embedding}
For every $p \in (1,\infty)$ and $s \in (0, 1)$, there exists a constant $C = C(N,s,\Omega)$ such
that $C_p = C_p(Q,s,\Omega) \to C \in (0,\infty)$ as $p \to \infty$ and $$\iint_{\Omega\times\Omega}\frac{|\bar{v}(x)-\bar{v}(y)|^p}{|y^{-1}\cdot x|^{Q+ps}}dy\leq C_p\|\nabla v\|_p,$$
 for every $v\in W_0^{1,p}(\Omega),$ where $\bar{v}$ is the extension
of $v$ 
 to $0$  
 in  $\Omega$.
\end{theorem}
\begin{proof}
Let $\bar{v}$ be an extension of $v\in W_0^{1,p}(\Omega)$ by $0$. This implies that $\|\nabla\bar{v}\|_{L^p(\mathbb{R}^n)}=\|\nabla\bar{v}\|_{L^p(\Omega)}$. Consider,
\begin{align}\label{eq1}
	\begin{split}
		\iint_{\mathbb{G}\times\mathbb{G}}\frac{|\bar{v}(x)-\bar{v}(y)|^p}{|y^{-1}\cdot x|^{Q+ps}}dxdy&=2\iint_{\Omega\times\Omega^c}\frac{|\bar{v}(x)-\bar{v}(y)|^p}{|y^{-1}\cdot x|^{Q+ps}}dxdy+\iint_{\Omega\times\Omega}\frac{|{v}(x)-{v}(y)|^p}{|y^{-1}\cdot x|^{Q+ps}}dxdy.
	\end{split}
\end{align}
We shall estimate only the first term since the second one can be estimated similarly. Let $B_{\rho}(x)$ denote the ball centered at $x\in\Omega$ with radius $\rho$. Then we have:
\begin{align}\label{eq2}
\begin{split}
\iint_{\Omega\times\Omega^c\cap B_{\rho}(x)}\frac{|\bar{v}(x)-\bar{v}(y)|^p}{|y^{-1}\cdot x|^{Q+ps}}dxdy\leq &\iint_{\Omega\times B_{\rho}(x)}\frac{|\bar{v}(x)-\bar{v}(y)|^p}{|y^{-1}\cdot x|^{Q+ps}}dxdy\\
\leq &\iint_{\Omega\times B_{\rho}(0)}\frac{|\bar{v}(x)-\bar{v}(L_z(x))|^p}{|z^{-1}|^{Q+ps}}dxdz\\
=&\iint_{\Omega\times B_{\rho}(0)}\frac{|\int_0^1\frac{d}{dt}\bar{v}(L_x\circ T_{t}(z))dt|^p}{|z|^{Q+ps}}dxdz\\
\leq &\iint_{\Omega\times B_{\rho}(0)}\int_0^1|\nabla_{\mathbb{G}}\bar{v}(L_x\circ T_{t}(z)))|^p\\
&\times \left|\frac{d}{dt}(L_x\circ T_{t}(z))\right|^p\frac{1}{|z|^{Q+ps}}dtdxdz\\
\leq &\omega_Q\|\nabla_{\mathbb{G}}\bar{v}\|_{L^p(\mathbb{G})}^2\int_0^\rho r^{p(1-s)-1}dr\\
=&\frac{\omega_Q}{p(1-s)}\rho^{p(1-s)}\|\nabla_{\mathbb{G}}\bar{v}\|_{L^p(\mathbb{G})}^p ,
\end{split}
\end{align}
where $L_z, T_t$, respectively, stand for translation by $z$ and multiplication by $t$, and $\omega_Q$ stands for the volume of a unit ball in a stratified Lie group whose homogeneous dimension is $Q$.

On the other hand, we have by the Poincar\'{e} inequality that
\begin{align}\label{eq3}
\begin{split}
\iint_{\Omega\times\Omega^c\cap B_{\rho}(x)^c}\frac{|\bar{v}(x)-\bar{v}(y)|^2}{|y^{-1}\cdot x|^{Q+ps}}dxdy\leq &\iint_{\Omega\times B_{\rho}(x)^c}\frac{|\bar{v}(x)-\bar{v}(y)|^p}{|y^{-1}\cdot x|^{Q+ps}}dxdy\\
\leq &\frac{\omega_Q}{ps}\rho^{-ps}\|w\|_{L^p(\Omega)}^p\leq \frac{\omega_Q(C(\Omega,p))^p}{ps}\rho^{-ps}\|\nabla_{\mathbb{G}} w\|_{L^p(\Omega)}^p.
\end{split}
\end{align}
This completes the proof of Theorem \ref{lemma_on_embedding}.
\end{proof}

\begin{remark}
$X$ is a real separable and reflexive Banach space (see Ghosh et al.  \cite{GVM1}).
\end{remark}
\begin{remark}
	For fine bounds on the best constants of the Sobolev embedding $W_0^{s,p}(\Omega)\hookrightarrow L^q(\Omega)$  in the Euclidean setup one may refer to Cassini-Du \cite{cassini}.
\end{remark}

\section{Operators $A:X\to X^*$ and $B:L^q(\Omega)\to (L^q(\Omega))^*$}\label{s22}

We define an  eigenpair for problem \eqref{main}.
\begin{definition}\label{weak_soln}
We say that $(\lambda,u)\in \mathbb{R}\times X\setminus\{0\}$ is an {\it eigenpair} for problem \eqref{main} if for every $\phi\in X$, we have
\begin{align}\label{eigenfunction}
\begin{split}
\int_{\Omega}\nabla_{\mathbb{G}}u\cdot\nabla_{\mathbb{G}}\phi dx+\iint_{\mathbb{G}\times\mathbb{G}}\frac{(u(x)-u(y))(\phi(x)-\phi(y))}{|y^{-1}\cdot  x|^{Q+2s}}dxdy&=\lambda\|u\|_{L^q(\Omega)}^{2-q}\int_{\Omega}|u|^{q-2}u\phi dx.
\end{split}
\end{align}
\end{definition}
Note that Theorem \ref{lemma_on_embedding} guarantees that eigenpairs are well-defined.
Next, we define the operators $A:X\to X^*$ by
\begin{align}\label{A}
\begin{split}
\langle Av,w\rangle&=\int_{\Omega}\nabla_{\mathbb{G}} v\cdot\nabla_{\mathbb{G}} wdx+\iint_{\mathbb{G}\times\mathbb{G}}\frac{(v(x)-v(y))(w(x)-w(y))}{|y^{-1}\cdot  x|^{Q+2s}}dxdy,
\end{split}
\end{align}
and $B:L^q(\Omega)\to (L^q(\Omega))^*$ by
\begin{align}\label{B}
\begin{split}
\langle Bv,w\rangle&=\int_{\Omega}|u|^{q-2}uwdx,
\end{split}
\end{align}
where the symbols $X^*$, $(L^q(\Omega))^*$ denote the dual of $X$, $L^q(\Omega),$ respectively. We prove the following theorem about their properties.

\begin{theorem}\label{aux_res_1}
The operators $A:X\to X^*$ and $B:L^q(\Omega)\to (L^q(\Omega))^*$ are continuous. Moreover, $A$ is bounded, coercive, and monotone.
\end{theorem}
\begin{proof}
{\bf Continuity:}~Suppose that $v_j\in X$ is such that $v_j\to v$ in the norm of $X$. Thus, a combination of Egoroff's theorem and the Sobolev embedding, we have up to a subsequence $\nabla_{\mathbb{G}} v_j(x)\to \nabla_{\mathbb{G}} v(x)$ a.e. in $\Omega$. We note that 
\begin{align}\label{eq_110_1}
\|\nabla_{\mathbb{G}} v_j\|_{L^{2}(\Omega)} \leq c,
\end{align} 
for some constant $c>0$ independent of $j$. Therefore, up to a subsequence, we have
\begin{align}\label{eq_110_2}
\nabla_{\mathbb{G}} v_j\rightharpoonup\nabla_{\mathbb{G}} v~\text{in}~L^{2}(\Omega).
\end{align}
In addition, we also have 
\begin{align}\label{eq_110_3}
\frac{v_j(x)-v_j(y)}{|y^{-1}\cdot  x|^{\frac{Q+2s}{2}}}\rightharpoonup \frac{v(x)-v(y)}{|y^{-1}\cdot  x|^{\frac{Q+2s}{2}}}~\text{on}~L^2(\mathbb{R}^n),
\end{align}
and since this weak limit is independent of the choice of the subsequence, it follows by \eqref{eq_110_2} and \eqref{eq_110_3} that
$$\underset{j\to\infty}{\lim}\langle Av_j,\phi\rangle=\langle Av,\phi\rangle,
\
\hbox{for every}
\
\phi\in X.
$$
 This proves that $A$ is continuous. It follows by a similar argument that $B$ is also continuous.\\

{\bf  Boundedness:}~By the Cauchy-Schwarz and the H\"{o}lder inequality, we have
\begin{align}\label{eq_110_4}
\begin{split}
\langle Av,\phi\rangle=&\int_{\Omega}\nabla_{\mathbb{G}} v\cdot\nabla_{\mathbb{G}}\phi dx+\iint_{\mathbb{G}\times\mathbb{G}}\frac{(v(x)-v(y))(\phi(x)-\phi(y))}{|y^{-1}\cdot  x|^{Q+2s}}dxdy\\
\leq & \|\nabla v\|_{2}\|\nabla \phi\|_2\\
&+\left(\iint_{\mathbb{G}\times\mathbb{G}}\frac{|v(x)-v(y)|^2}{|y^{-1}\cdot  x|^{Q+2s}}dxdy\right)^{1/2}\left(\iint_{\mathbb{G}\times\mathbb{G}}\frac{|\phi(x)-\phi(y)|^2}{|y^{-1}\cdot  x|^{Q+2s}}dxdy\right)^{1/2}\\
\leq & \left(\|\nabla v\|_2+\left(\iint_{\mathbb{G}\times\mathbb{G}}\frac{|v(x)-v(y)|^2}{|y^{-1}\cdot  x|^{Q+2s}}dxdy\right)^{1/2}\right)\|\phi\|_{X}\\
\leq & \left(\|\nabla v\|_2^2+\left(\iint_{\mathbb{G}\times\mathbb{G}}\frac{|v(x)-v(y)|^2}{|y^{-1}\cdot  x|^{Q+2s}}dxdy\right)^{2}\right)^{1/2}\|\phi\|_{X}=\|v\|_X\|\phi\|_X,
\end{split}
\end{align}	
hence,
$$\|Av\|_{X^*}=\underset{\|\phi\|_X\leq 1}\sup|\langle Av,\phi\rangle|\leq \|v\|_X\|\phi\|_X\leq \|v\|_X,$$
therefore $A$ is bounded.\\

{\bf Coercivity:}~We notice that 
\begin{align}\label{eq_110_5}
\langle Av,v\rangle&=\int_{\Omega}|\nabla_{\mathbb{G}} v|^2dx+\iint_{\mathbb{G}\times\mathbb{G}}\frac{|v(x)-v(y)|^2}{|y^{-1}\cdot  x|^{Q+2s}}dxdy=\|v\|_X^2,
\end{align}
which provides that $A$ is coercive.\\

{\bf Monotonicity:}~For $u\in X$, let
$$\mathcal{O}(u(x,y))=u(x)-u(y),~~d\mu=\frac{dxdy}{|y^{-1}\cdot  x|^{Q+2s}}.$$
Considering the following
\begin{align}\label{eq_110_6}
\begin{split}
\langle Av-A\phi,v-\phi\rangle=&\int_{\Omega}|\nabla_{\mathbb{G}} v-\nabla_{\mathbb{G}}\phi|^2dx\\
&+\iint_{\mathbb{G}\times\mathbb{G}}(\mathcal{O}(v(x,y))-\mathcal{O}(w(x,y)))((v(x)-w(x))-(v(y)-w(y)))d\mu
\geq  0,
\end{split}
\end{align}
we conclude that $A$ is monotone. This completes the proof of Theorem \ref{aux_res_1}.
\end{proof}

\section{Proof of Theorem \ref{main_result_3}}\label{s3}
We shall need the following results in the sequel.

\begin{theorem}\label{main_result_1}
Let $0<s<1$ and $1<q<2^*$. Then the following properties hold:
\begin{enumerate}
\item There exists a sequence $(w_j)\subset X\cap L^q(\Omega)$ such that $\|w_j\|_{L^q(\Omega)}=1$ and for every $\phi\in X$, we have
\begin{align}\label{conclusion1}
\begin{split}
\int_{\Omega}\nabla_{\mathbb{G}}w_j\cdot\nabla_{\mathbb{G}}\phi dx&+\iint_{\mathbb{G}\times\mathbb{G}}\frac{(w_j(x)-w_j(y))(\phi(x)-\phi(y))}{|y^{-1}\cdot  x|^{Q+2s}}dxdy\\
&=\mu_n\|w_j\|_{L^q(\Omega)}^{2-q}\int_{\Omega}|w_j|^{q-2}w_j\phi dx,
\end{split}
\end{align}
where $$\mu_j\geq \lambda:=\inf\{\int_{\Omega}|\nabla_{\mathbb{G}}u|^2dx+\iint_{\mathbb{G}\times\mathbb{G}}\frac{|u(x)-u(y)|^2}{|y^{-1}\cdot  x|^{Q+2s}}dxdy:u\in X\cap L^q(\Omega), \|u\|_{L^q(\Omega)}=1\}.$$
\item The sequences $(\mu_j)$ and $(\|w_{j+1}\|^p)$ satisfying \eqref{conclusion1} are nonincreasing and they converge to the same limit $\mu$, which is bounded below by $\lambda$. Moreover, there exists a subsequence $(n_j)$ such that both $(w_{j_i})$ and $(w_{j_{i+1}})$ converge in $X$ to the same limit $w\in X\cap L^q(\Omega)$ with $\|w\|_{L^q(\Omega)}=1,$ and $(\mu,w)$ is an eigenpair for problem \eqref{main}.
\end{enumerate}
\end{theorem}
\begin{proof}
One can follow the argument of the proof of Garain-Ukhlov  \cite[Theorem $2.2$]{PG1}.
\end{proof}
\begin{theorem}\label{main_result_2}
Let $0<s<1$ and $1<q<2^*$. Suppose that $(u_j)\subset X\cap L^q(\Omega)$ is such that $\|u_j\|_{L^q(\Omega)}=1$ and $\underset{j\to\infty}\lim\|u_j\|^p=\lambda$. Then there exists a subsequence $(u_{j_i})$ which converges weakly in $X$ to $u\in X\cap l^q(\Omega)$ with $\|u\|_{L^q(\Omega)}=1$ such that $$\lambda=\int_{\Omega}|\nabla_{\mathbb{G}}u|^pdx+\iint_{\mathbb{G}\times\mathbb{G}}\frac{|u(x)-u(y)|^{2}}{|y^{-1}\cdot  x|^{Q+2s}}dxdy.$$
Moreover, $(\lambda,u)$ is an eigenpair for problem \eqref{main} and every associated eigenfunction of $\lambda$ is a scalar multiple of the vector at which $\lambda$ is reached.
\end{theorem}
\begin{proof}
One can follow the argument of the proof of  Garain-Ukhlov  \cite[Theorem 2.2]{PG1}. However,  we prefer to apply Ercole \cite[Proposition $2$]{ercole1} in place of Ercole \cite[Theorem $1$]{ercole1}.
\end{proof}

We are now in a position to prove our first main result.
\begin{proof}[Proof of Theorem \ref{main_result_3}]
Since \eqref{main} is homogeneous, we can without loss of generality,
 let $\|u\|_q=1$. Furthermore, for every $k\geq 1,$ let $\Omega_k:=\{x\in\Omega:u(x)>k\}$ and $v:=(u-k)_+$ be a test function for \eqref{eigenfunction}. With these choices we obtain
\begin{align}\label{eig1}
\begin{split}
\int_{\Omega_k}|\nabla_{\mathbb{G}}u|^2dx &+\iint_{\mathbb{G}\times\mathbb{G}}\frac{(u(x)-u(y))((u-k)_+(x)-(u-k)_+(y))}{|y^{-1}\circ  x|^{Q+2s}}dxdy\\
&=\lambda\int_{\Omega}u^{q-1}(u-k)_+dx.
\end{split}
\end{align}
The second term on the left hand side of \eqref{eig1} is nonnegative and hence
\begin{align}\label{eig2}
\begin{split}
\int_{\Omega_k}|\nabla_{\mathbb{G}}u|^2dx&\leq \lambda\int_{\Omega}u^{q-1}(u-k)_+dx.
\end{split}
\end{align}
We split the proof into two parts.\\

{\bf First case:~ $q\leq 2$}. Since $k\geq 1$, we have $u^{q-1}\leq u$ on $\Omega_k$.  From \eqref{eig2}  we have
\begin{align}\label{eig3}
\begin{split}
\int_{\Omega_k}|\nabla_{\mathbb{G}}u|^2dx&\leq \lambda\int_{\Omega}u(u-k)dx=\lambda\int_{\Omega_k}(2(u-k)^2+2k(u-k))dx,
\end{split}
\end{align}
where the last inequality is obtained using $a+b\leq2(a+b)$.  Using the Sobolev inequality we get
\begin{align}\label{eig4}
\begin{split}
(1-2S\lambda|\Omega_k|^{2/N})\int_{\Omega_k}(u-k)^pdx&\leq k\lambda|\Omega_k|^{2/N}\int_{\Omega_k}(u-k)dx,
\end{split}
\end{align}
where $S$ is the best Sobolev constant. Note that $\| u\|_{1}\geq k|\Omega_k|$ and therefore, for every $k\geq k_* = $  $(4S\lambda)^{\frac{N}{2}} \| u\|_{1},$ we have
$$2S\lambda|\Omega(k)|
  ^\frac{2}{N} \leq \frac{1}{2}. $$\\
Using  \eqref{eig3}, for every $ k\geq \max \{ k_*, 1\},$ we obtain
\begin{equation}\label{Eq 0.3}
	\int _{L(k)} (u-k)^2 dx \leq 4S\lambda k |\Omega_k|^\frac{2}{N} \int_{\Omega_k} (u-k)dx,
\end{equation}
and invoking the H\"{o}lder inequality and the estimate \eqref{Eq 0.3}, we find 
\begin{equation}\label{Eq 0.4}
	\int_{\Omega_k} (u-k)dx \leq  (4S\lambda)k|\Omega_k|^{1+\frac{2}{N}},
\end{equation}
so taking into account \eqref{Eq 0.4}, we conclude by invoking  Ladyzhenskaya-Ural'tseva  \cite[Lemma $5.1$]{lady0}
(which holds irrespective of the group structure on $\mathbb{R}^n$), to get $ u \in L^{\infty} (\Omega).$ \\

{\bf Second case: $q>2.$} Using the inequality
 
 $$ (a+b)^{q-1} \leq 2^{q-1} (a^{q-1}+ b^{q-1}),
 \
 \hbox{for every}
 \
 a,b \geq 0,
 $$
 in
  $(4.1),$   we get \\
\begin{equation}\label{Eq 0.5}
	\int_{\Omega_k} |\nabla_{\mathbb{G}} u|^2 dx \leq \lambda \int_{\Omega_k} (2^{q-1} (u-k)^q) + 2^{q-1}k^{q-1}(u-k))dx
\end{equation}
and using the Sobolev inequality with $r = q$ in  estimate \eqref{Eq 0.5}, we find 
\begin{equation}\label{Eq 0.6}
	\left(\int_{\Omega_k} (u-k)^q dx \right) ^\frac{2}{q} \leq S \lambda |\Omega_k|^{2(\frac{1}{q}-\frac{1}{2}+ \frac{1}{N} )} \int_{\Omega_k} (2^{q-1}(u-k)^q + 2^{q-1} k^{q-1} (u-k))dx,\\
\end{equation}
where $S>0$ is the Sobolev constant. Since\\
$$ \int_{\Omega_k} (u-k)^q dx \leq \|u\|^q _{q} =1$$\\
and $ q>2,$ the quantity on the left-hand side of \eqref{Eq 0.6} can be estimated from below as follows:
\begin{equation}\label{Eq 0.7}
	\left(\int_{\Omega_k} (u-k)^q dx \right) ^\frac{2}{q} = \left(\int_{\Omega_k} (u-k)^q dx \right) ^{\frac{2-q}{q} +1} \geq \int_{\Omega_k} (u-k)^q dx.
\end{equation}
Using 
\eqref{Eq 0.7}
 in \eqref{Eq 0.6}, 
 we get 
\begin{equation}\label{Eq 0.8}
	\begin{aligned}
		&(1- S \lambda 2^{q-1} |\Omega_k|^{p(\frac{1}{q}-\frac{1}{2}+ \frac{1}{N} )})  \int_{\Omega_k} (u-k)^q dx\\
		& \leq S \lambda 2^{q-1} k^{q-1}  |\Omega_k|^{2(\frac{1}{q}-\frac{1}{2}+ \frac{1}{N} )}
		 \int_{\Omega_k} (u-k) dx.
	\end{aligned}
\end{equation}
Let $$\alpha = 2\left(\frac{1}{q} -\frac{1}{2} + \frac{1}{N}\right),$$ and note that it is positive since $1<q<2^\ast .$ Choose $K_1 =(S\lambda 2^q)^{\frac{1}{\alpha}}\|u\|_{1}.$ Then since $ k|\Omega_k| \leq \|u\|_{1}, $ for every $k \geq k_*,$ we have
$$ S \lambda 2^{q-1} |\Omega_k|^{\alpha} \leq \frac {1}{2}. $$
Using this property in \eqref{Eq 0.8}, we obtain  
\begin{equation}\label{Eq 0.9}
	\int_{L(k)} (u-k)^q dx \leq  S \lambda 2^q k^{q-1} |L(k)|^{\alpha}  \int_{L(k)} (u-k) dx,
\end{equation}
so by the H\"{o}lder inequality and the estimate \eqref{Eq 0.9}, we arrive at the following estimate
\begin{equation}\label{Eq 1.0}
	\int_{L(k)} (u-k) dx \leq   (S \lambda 2^q)^{\frac{1}{q-1}}k|L(k)|^{1+\frac{\alpha}{q-1}}.
\end{equation}
Taking into account \eqref{Eq 1.0} and invoking Ladyzhenskaya-Ural'tseva  \cite[Lemma $5.1$]{lady0}, we get $u \in L^{\infty} (\Omega).$

The other assertion follows by Garain-Kinnunen  \cite[Theorem $8.4$]{PG0} which considers the following problem:
\begin{align}\label{juha1}
\begin{split}
-\Delta_pu+(-\Delta_{p,K})^su=&0~\text{in}~\Omega,
\end{split}
\end{align}
where $1<p<\infty$ and $$(-\Delta_{p,K})^su=\text{P.V.}
\int_{\mathbb{R}^n}|u(x)-u(y)|^{p-2}(u(x)-u(y))K(x,y)dy,$$ here P.V. denotes the principle value. The weak Harnack inequality proved in  Garain-Kinnunen \cite{PG0} also works for the
 Dirichlet boundary condition $u|_{\partial\Omega}=0$. In addition, we note that the proof of  Garain-Kinnunen \cite[Theorem 8.4]{PG0} works even for a sub-Laplacian $\mathcal{L}$ that acts on functions defined on a stratified Lie group since the group structure does not affect the proof of the result. Let us denote $v$ to be a nonnegative solution of problem \eqref{main}. On subtracting the weak formulation of \eqref{juha1} (for $p=2$) from \eqref{main} we get
\begin{align}\label{WF_Difference0}
	\begin{split}
		\int_{\Omega}\nabla_{\mathbb{G}}(v-u)&\cdot\nabla_{\mathbb{G}}(v-u)^-dx\\
		&+\iint_{\Omega\times\Omega}((v-u)(x)-(v-u)(y))((v-u)^-(x)-(v-u)^-(y))|y^{-1}\cdot x|^{-Q-2s}dx\\
		=&\lambda\|u\|_q^{2-q}\int_{\Omega}v^{q-1}(v-u)^-dx.
	\end{split}
\end{align}
This implies
\begin{align}\label{WF_Difference1}
	\begin{split}
		0\geq-\int_{\Omega}|\nabla_{\mathbb{G}}(v-u)^-|^2dx &-\iint_{\Omega\times\Omega}\frac{|(v-u)^-(x)-(v-u)^-(y)|^2}{|y^{-1}\cdot x|^{Q+2s}}dx\\
		=&\lambda\|u\|_q^{2-q}\int_{\Omega}v^{q-1}(v-u)^-dx\geq 0,
	\end{split}
\end{align}
hence the Lebesgue measure of the set $\{x\in\Omega:v(x)\leq u(x)\}$ is zero and therefore
 $v\geq u$ a.e. in $\Omega$. By  \cite{PG0}, we already have $u>0$ on any compact subset of $\Omega$ and hence we have the same conclusion for $v$.
This completes the proof of Theorem \ref{main_result_3}.
\end{proof}

\section{Proof of Theorem \ref{th2}}\label{s4}
A natural question to raise at this juncture is what happens when $q=2^*$? Of course, it is impossible to figure out the answer for a general stratified Lie group involving a fractional Laplacian - owing to the unavailability of the derivative of the distance function. However, in order to be able to answer this question, we shall study the {\it Brezis-Nirenberg} type of problem on a Heisenberg group $\mathbb{G}:=\mathbb{H}^n $ (see Molica Bisci-Repov\v{s} \cite{bisci}) for which we shall establish a {\it Pohozaev-like identity}.
 The problem is as follows:
\begin{equation}\label{Brezis-Nirenberg}
	\left\{ \begin{array}{ll}-\mathcal{L} u+(-\Delta)^su=\lambda\|u\|_{2^*}^{2-2^*}|u|^{2^*-2}u=:g(u)~~~\text{on}~\Omega\\
u=0~~~\text{on}~\mathbb{R}^n\setminus\Omega.\end{array} 
	\right.
\end{equation}
Let $$G(u):=\int_0^ug(t)dt$$ be the primitive of the continuous function $g$. 
The vector fields $$X_i=\frac{\partial}{\partial x_i}-\frac{y_i}{2}\frac{\partial}{\partial z}, Y_i=\frac{\partial}{\partial y_i}+\frac{x_i}{2}\frac{\partial}{\partial z}, Z=\frac{\partial}{\partial z},
\
\hbox{where}
\
i=1,2,\cdots,n
$$ 
generate
 the Lie algebra corresponding to the Heisenberg group $\mathbb{H}^n$ of topological dimension $2n+1$. The corresponding group law for $(a,b,c),(a',b'c')\in\mathbb{R}^{2n}\times\mathbb{R}$ is 
	$$(a,b,c)\circ (a',b'c')^{-1}=(a+a',b+b',c+c'+2^{-1}(ab'-ba')).$$
Under this group law, the inverse is $(a,b,c)^{-1}=(-a,-b,-c)$ and the identity is $(0,0,0)$. The distance function for a Heisenberg group is defined as 
$$|(a,b,c)|=[(|a|^2+|b|^2)^2+|c|^2]^{1/4}.$$
However, before we establish a Pohozaev-like identity, we need to prove Theorem \ref{identity1}.

\begin{theorem}\label{identity1}
Suppose that $u_1,u_2\in W^{1,1}({\mathbb{H}^n})$ have disjoint compact supports, say $\Omega_1, \Omega_2,$ respectively. Then
\begin{align}\label{id1}
\begin{split}
\int_{\Omega_1}(\bar{x}\cdot\nabla_{\mathbb{H}^n} u_1)(-\Delta)^su_2dx+&\int_{\Omega_2}(\bar{x}\cdot\nabla_{\mathbb{H}^n} u_2)(-\Delta)^su_1dx\\
&=\left(\frac{2s-Q}{2}\right)\int_{\Omega_1}u_1(-\Delta)^su_2dx+\left(\frac{2s-Q}{2}\right)\int_{\Omega_2}u_2(-\Delta)^su_1dx,
\end{split}
\end{align}
where $\bar{x}=(x_1^{(1)},x_2^{(1)},\cdots,x_{N_1}^{(1)},\cdots,rx_1^{(r)},rx_2^{(r)}, \cdots,rx_{N_r}^{(r)})$
and
 $\nabla_{\mathbb{H}^n}$ is the subgradient corresponding to the Heisenberg group $\mathbb{H}^n$. 
\end{theorem}

\begin{proof}[Proof of Theorem \ref{identity1}]
We first claim that 
\begin{align}\label{id2}
(-\Delta)^s(\bar{x}\cdot\nabla_{\mathbb{H}^n} u_j)&=\bar{x}\cdot(-\Delta)^s\nabla_{\mathbb{H}^n} u_j+2s(-\Delta)^su_j~\text{in}~\mathbb{H}^n\setminus\Omega_i,
\
\hbox{for every}
\
i=1,2.
\end{align}
Clearly, we have $u_j\equiv 0$ on $\mathbb{H}^n\setminus\Omega_i$ and hence using the definition of $(-\Delta)^s$ for $x\in \mathbb{H}^n\setminus\Omega_i,$ we obtain
\begin{align}\label{id3}
\begin{split}
(-\Delta)^s(\bar{x}\cdot\nabla_{\mathbb{H}^n} u_i(x))=&c_{N,s}\int_{\Omega_i}\frac{-\bar{y}\cdot\nabla_{\mathbb{H}^n} u_i(y)}{|y^{-1}\circ  x|^{Q+2s}}dy\\
=& c_{N,s}\int_{\Omega_i}\frac{\sum_{k=1}^r\sum_{l=1}^{N_j}k(x_l^{(k)}-y_l^{(k)})\frac{\partial}{\partial y_l^{(k)}}u_i(y)}{|y^{-1}\circ  x|^{Q+2s}}dy\\
&+c_{N,s}\int_{\Omega_i}\frac{\sum_{k=1}^r\sum_{l=1}^{N_j}kx_l^{(k)}\frac{\partial}{\partial y_l^{(k)}}u_i(y)}{|y^{-1}\circ  x|^{Q+2s}}dy\\
=&  c_{N,s}\int_{\Omega_i}-\sum_{k=1}^r\sum_{l=1}^{N_j}k\frac{(x_l^{(k)}-y_l^{(k)})}{|y^{-1}\circ  x|^{Q+2s}}\frac{\partial}{\partial y_l^{(k)}}u_i(y)+\bar{x}\cdot (-\Delta)^s\nabla_{\mathbb{H}^n} u_i(x)\\
=&  c_{N,s}\int_{\Omega_i}-\frac{2s}{|y^{-1}\circ  x|^{Q+2s}}u_i(y)+\bar{x}\cdot (-\Delta)^s\nabla_{\mathbb{H}^n} u_i(x)\\
=& 2s(-\Delta)^su_i(y)+\bar{x}\cdot (-\Delta)^s\nabla_{\mathbb{H}^n} u_i(x),
\
\hbox{for every}
\
i=1,2,
\end{split}
\end{align}
hence
\begin{align}\label{equ11}
\begin{split}
\int_{\Omega_1}(\bar{x}\cdot\nabla_{\mathbb{H}^n} u_1)(-\Delta)^su_2dx=&-Q\int_{\Omega_1}u_1(-\Delta)^su_2dx-\int_{\Omega_1}u_1\bar{x}\cdot\nabla_{\mathbb{H}^n} (-\Delta)^su_2dx,
\end{split}
\end{align}
so integrating by parts along with \eqref{id3}, we obtain
\begin{align}\label{equ12}
\begin{split}
\int_{\Omega_1}u_1\bar{x}\cdot\nabla_{\mathbb{H}^n}(-\Delta)^su_2dx=&\int_{\Omega_1}u_1(-\Delta)^s(\bar{x}\cdot\nabla_{\mathbb{H}^n} u_2)dx-2s\int_{\Omega_1}u_1(-\Delta)^su_2dx\\
=& \int_{\Omega_1}(-\Delta)^su_1(\bar{x}\cdot\nabla_{\mathbb{H}^n} u_2)dx-2s\int_{\Omega_1}u_1(-\Delta)^su_2dx,
\end{split}
\end{align}
therefore by \eqref{equ11} and \eqref{equ12},
\begin{align}\label{equ13}
\begin{split}
\int_{\Omega_1}(\bar{x}\cdot\nabla_{\mathbb{H}^n} u_1)(-\Delta)^su_2dx=& -\int_{\Omega_1}(-\Delta)^su_1(\bar{x}\cdot\nabla_{\mathbb{H}^n} u_2)dx+(2s-Q)\int_{\Omega_1}u_1(-\Delta)^su_2dx,
\end{split}
\end{align}
thus, by again integrating by parts,
\begin{align}\label{equ14}
\begin{split}
\int_{\Omega_1}u_1(-\Delta)^su_2dx=&\frac{1}{2}\int_{\Omega_1}u_1(-\Delta)^su_2dx+\frac{1}{2}\int_{\Omega_2}u_2(-\Delta)^su_1dx,
\end{split}
\end{align}
and the equality \eqref{id1} follows. This completes the proof of Theorem \ref{identity1}
\end{proof}
\begin{proof}[Proof of Theorem \ref{th2}]
 We shall
 follow the standard technique of deriving the identity - by multiplying the PDE in \eqref{Brezis-Nirenberg} by
 $$Zu:=\sum_{j=1}^r\sum_{i=1}^{n_j}ix_i^{(j)}\frac{\partial u}{\partial x_i^{(j)}}.$$
  We  note of  that $\text{div}Z=Q$. 
  Invoking Bonfiglioli et al.  \cite[Proposition 1.6.1]{AB3} and integrating by parts, we get
\begin{align}\label{pohozaev_identity1}
\begin{split}
-\int_{\Omega}\Delta_{\mathbb{G}}uZudx+\int_{\Omega}(-\Delta)^suZudx=&\int_{\Omega}g(u)Zudx\\
=&\int_{\Omega}\sum_{i=1}^{n_1}ix_i\frac{\partial}{\partial x_i}G(u)dx
=-\frac{Q}{2}\int_{\Omega}G(u)dx.
\end{split}
\end{align}
Combining \eqref{pohozaev_identity1} with  Theorem \ref{identity1} and  Louidice  \cite[Equation (4.15) in Theorem $4.1$]{louidice}, the identity can be stated as follows:
\begin{align}\label{pohozaev_identity_2}
\begin{split}
\frac{Q}{2}\int_{\Omega}G(u)dx-\left(\frac{Q-2}{2}\right)\int_{\Omega}|\nabla_{_{\mathbb{H}^n}}u|^2dx&-\left(\frac{Q-2s}{2}\right)\int_{\Omega}|(-\Delta)^su|^2dx\\
=&\frac{1}{2}\int_{\partial\Omega}|\nabla_{_{\mathbb{H}^n}}u|^2\langle Z,\hat{n}\rangle dS.
\end{split}
\end{align}
This completes the proof of Theorem \ref{th2}. 
\end{proof}
\begin{remark}\label{pohozaev_example}
Using Theorem \ref{th2}, one can show that problem \eqref{Brezis-Nirenberg} when defined on a star-shaped domain $\Omega,$ has no nontrivial solutions for any $\lambda<0$. This is because equation \eqref{pohozaev_example} yields
\begin{align}\label{counter_example}
\begin{split}
0<\frac{1}{2}\int_{\partial\Omega}|\nabla_{_{\mathbb{H}^n}}u|^2\langle Z,\hat{n}\rangle dS=&\frac{Q}{2}\int_{\Omega}G(u)dx-\frac{Q}{2}\int_{\Omega}g(u)udx+\left(\int_{\Omega}|\nabla_{_{\mathbb{H}^n}}u|^2dx+s\int_{\Omega}|(-\Delta)^su|^2dx\right)\\
\leq& \frac{Q}{2}\int_{\Omega}G(u)dx-\frac{Q}{2}\int_{\Omega}g(u)udx+\left(\int_{\Omega}|\nabla_{_{\mathbb{H}^n}}u|^2dx+\int_{\Omega}|(-\Delta)^su|^2dx\right)\\
=& \frac{Q}{2}\int_{\Omega}G(u)dx-\frac{Q}{2}\int_{\Omega}g(u)udx+\int_{\Omega}g(u)udx\\
=& \frac{Q}{2}\int_{\Omega}G(u)dx+\left(1-\frac{Q}{2}\right)\int_{\Omega}g(u)udx< 0.
\end{split}
\end{align}
This is absurd and hence there exists no nontrivial solution to the problem if $\lambda<0$ and the domain $\Omega$ is star-shaped.
\end{remark}

\subsection*{Acknowledgements}
The first author was supported by the National Board for Higher Mathematics (NBHM), Department of Atomic Energy (DAE) India, [02011/47/2021/NBHM(R.P.)/R\&D II/2615]. The third author was supported by the Slovenian Research and Innovation Agency program P1-0292 and grants J1-4031, J1-4001, N1-0278, N1-0114, and N1-0083.
We thank the referees for their comments and suggestions.

\end{document}